\documentclass[11pt]{amsart}

\usepackage[utf8]{inputenc}
\usepackage{mathtools}
\usepackage{amsmath}
\usepackage{amsfonts}
\usepackage{fancyhdr}
\usepackage{amssymb}
\usepackage{amsthm}
\usepackage[margin= 1 in]{geometry}
\usepackage{subcaption}
\usepackage[colorlinks=true, pdfstartview=FitV, linkcolor=blue, citecolor=blue, urlcolor=blue]{hyperref}

\usepackage{xcolor}

\usepackage{tikz}
\usetikzlibrary{cd}
\usepackage{ytableau}
\usepackage{makecell}
\usepackage{bm}

\usepackage[colorinlistoftodos]{todonotes}

\newcommand{\defn}[1]{{\color{darkred}\emph{#1}}} 
\definecolor{darkblue}{rgb}{0.0,0,0.7} 
\definecolor{darkred}{rgb}{0.7,0,0} 
\definecolor{darkgreen}{rgb}{0, .6, 0} 

\newcommand{\rank}{\mathsf{rank}}
\newcommand{\inv}{\mathsf{inv}}
\newcommand{\PF}{\mathsf{PF}}
\newcommand{\tsum}{\mathsf{sum}}
\newcommand{\LT}{\mathsf{LT}}

\newtheorem{theorem}{Theorem}[section]

\newtheorem{corollary}[theorem]{Corollary}
\newtheorem{conjecture}[theorem]{Conjecture}

\newtheorem{proposition}[theorem]{Proposition}

\newtheorem{remark}[theorem]{Remark}
\numberwithin{equation}{section}

\title{A Note on the Log-Concavity of Parking Functions}

\author[J.~Pappe]{Joseph Pappe}
\address[J. Pappe]{Department of Mathematics, Colorado State University, Fort Collins, CO, U.S.A.}
\email{joseph.pappe@colostate.edu}

\date{\today}

\date{\today}

\begin{document}

\begin{abstract}
We settle a conjecture of B\'ona regarding the log-concavity of a certain statistic on parking functions by utilizing recent log-concavity results on matroids. This result allows us to also prove that connected, labeled graphs graded by their number of edges are log-concave. Furthermore, we generalize these results to $G$-parking functions.
\end{abstract}

\maketitle 

\section{Introduction}

A \defn{parking function} of length $n$ is a sequence $(a_1, a_2, \ldots, a_n)$ of positive integers such that its nondecreasing rearrangement $b_1 \leq b_2 \leq \ldots \leq b_n$ satisfies $b_{i} \leq i$ for all $1\leq i \leq n$. Originally defined by Konheim and Weiss ~\cite{KonWei.1966} in their study of linear probes of random hash functions, this combinatorial object has appeared in many combinatorial and algebraic contexts. See for instance the survey of Yan ~\cite{Yan.2015}. The most pertinent connection for the purpose of this paper is their relationship with labeled trees. In their original paper ~\cite{KonWei.1966}, Konheim and Weiss established that the cardinality of $\PF_n$, the set of parking functions of length $n$, is given by $(n+1)^{n-1}$. As such, parking functions of length $n$ are in bijection with labeled trees on $n+1$ vertices.

A natural statistic to define on a parking function $\pi = (a_1,a_2,\ldots,a_n) \in \PF_n$ is the sum of its entries denoted $\tsum(\pi) = a_1+a_2+\ldots+a_n$. This statistic has been studied under many equivalent formulations. For instance, it is essentially equivalent to the displacement ~\cite{Yan.2015} and area ~\cite{Haglund.2008} statistic on parking functions, the latter of which being pivotal to the study of the space of diagonal harmonics ~\cite{CarMel.2018, Haiman.1994}. With respect to labeled trees, the $\tsum$ statistic on parking functions is related to the number of inversions on labeled trees. Given a tree $T$ with vertex set $\{0, 1, \ldots, n\}$, an \defn{inversion} is an ordered pair $(i,j)$ of vertices such that $0 < i < j$ and $j$ lies on the unique path from $i$ to $0$. Denote by $\inv(T)$ the total number of inversions in $T$ and let $I_{n+1}(x) \coloneqq \sum_{T \in \LT_{n+1}} x^{\inv(T)}$  where $\LT_{n+1}$ is the set of labeled trees on the vertex set $\{0, 1, \ldots, n\}$. Inversions on labeled trees are known to be related to the sum statistic on parking functions as follows
\begin{equation} ~\label{eq:pftree}
	P_n(x) \coloneqq \sum_{\pi \in \PF_n} x^{\tsum(\pi)} = x^{\binom{n+1}{2}} I_{n+1}(x^{-1}).
\end{equation}
 For a bijective proof of this formula see ~\cite{Yan.2015}.
 
A sequence of real numbers $a_0, a_1, \ldots, a_n$ is said to be \defn{log-concave} if $(a_i)^2 \geq a_{i-1}a_{i+1}$ for all $0 < i < n$. Log-concave sequences are of interest as any log-concave sequence of positive real numbers is also \defn{unimodal}, i.e. $a_0 \leq a_1 \leq \ldots \leq a_m \geq a_{m+1} \geq \ldots \geq a_n$ for some $0\leq m \leq n$. We say that a generating function is log-concave (resp. unimodal) if its sequence of coefficients is log-concave (resp. unimodal). In a grant proposal from 2000, B\'ona conjectured the following:
\begin{conjecture} \cite{Bona} \label{conj.main}
The polynomial $P_n(x)$ is log-concave for all $n\geq 1$.
\end{conjecture}

In recent times, there has a been a interest in the log-concavity of sequences stemming from matroid theory. Within the past decade, several of the largest outstanding log-concavity conjectures for matroid invariants have been resolved, see Eur's survey ~\cite{Eur.2024} for a more comprehensive story. Given these recent results, the log-concavity of the polynomial $I_{n}(x)$ will follow from associating the polynomial to an invariant of an appropriate matroid. From this, B\'ona's conjecture then follows by Equation ~\ref{eq:pftree}. After settling the log-concavity of $P_n(x)$, we show that the edge enumerator of labeled, connected graphs is also log-concave. Moreover, we generalize these results to the setting of Postnikov and Shapiro's $G$-parking functions ~\cite{PostSha.2004}.

The paper is organized as follows. In Section ~\ref{section.background} we review the necessary background information on matroids and relevant log-concavity results. The main results are then proved in Section ~\ref{section.results}.

\subsection*{Acknowledgments}
The author would like to thank Mikl\'os B\'ona, Maria Gillespie, and Eugene Gorsky for helpful discussions.

\section{Background on Matroids}
\label{section.background}

A \defn{matroid} $M$ consists of a pair $(E, \mathcal{I})$ where $E$ is a finite set of objects and $\mathcal{I}$ is a nonempty collection of subsets of $E$ satisfying the following conditions:
\begin{enumerate}
	\item (hereditary property) if $J \in \mathcal{I}$ and $I\subseteq J$, then $I\in\mathcal{I}$
	\item (exchange property) if $I,J \in \mathcal{I}$ and $|I| < |J|$, then there exists $j \in J\backslash I$ such that $I\cup\{j\} \in \mathcal{I}$
\end{enumerate}
The set $E$ is known as the \defn{ground set} of $M$ and $\mathcal{I}$ is the family of \defn{independent sets} of $M$. An important family of matroids comes from graphs. Given a graph $G$ with edge set $E_{G}$, let $\mathcal{I}_{G}$ be the collection of acyclic subsets of $E$ (i.e. $\mathcal{I}$ is the collection of forests of $G$). It is a classical result that for any graph $G$, $(E_G, \mathcal{I}_G)$ is always a matroid. We define this matroid to be the \defn{graphical matroid} of the graph $G$ and denote it by $M(G) = (E_G, \mathcal{I}_G)$.

A matroid $M$ is said to be \defn{representable} over a field $\mathbb{F}$ if it can be realized as $(E',\mathcal{I}')$ where $E'$ is a set of vectors in a vector space over $\mathbb{F}$ and $\mathcal{I}'$ is the collection of linearly independent subsets of $E'$. We will exploit the following relationship between graphical and representable matroids.

\begin{proposition} \cite[Proposition 5.1.2]{Oxley.2011}
	Let $G$ be a graph and $M(G)$ its graphical matroid. Then $M(G)$ is representable over every field.
\end{proposition}

Given the hereditary property of matroids, it suffices to define a matroid by specifying its ground set $E$ and maximal independent sets, known as the \defn{bases} of the matroid. The \defn{dual matroid} of $M = (E, I)$, denoted $M^{\perp}$, is then defined to be the matroid on the same ground set $E$ with basis elements given by $\{E\backslash B \colon B \text{ is a basis of } M\}$. 

\begin{proposition} \cite[Corollary 2.2.9]{Oxley.2011}
	If a matroid $M$ is representable over a field $\mathbb{F}$, then $M^{\perp}$ is also representable over $\mathbb{F}$.
\end{proposition}

For a given subset $S\subseteq E$, we define its \defn{rank} with respect to the matroid $M = (E, \mathcal{I})$ to be $\rank_{M}(S) = \max\{|I|\colon I \subseteq S \text{ and } I \in \mathcal{I} \}$. Furthermore, we define the rank $r_M$ of a matroid $M$ to be the rank of any basis of $M$, which is well-defined by the exchange property. Using the rank function, we define the \defn{Tutte polynomial} of $M$, denoted $T_M(x,y)$ as follows
\begin{equation}
	T_M(x,y) = \sum_{S\subseteq E} (x-1)^{r_M-\rank_M(S)}(y-1)^{|S|-\rank_M(S)}.
\end{equation}

This definition of the Tutte polynomial generalizes the original definition on graphs. More specifically, we have that the Tutte polynomial of the graphic matroid $M(G)$ is the same as the original Tutte polynomial on the graph $G$, see ~\cite{Biggs.1993} for more information on the Tutte polynomial of graphs. In addition, the Tutte polynomial interacts nicely with the dual of a matroid, namely, we have for all matroids $M$
\begin{equation}  \label{eq:dual}
T_{M}(x,y) = T_{M^{\perp}}(y,x).
\end{equation}

Recent results on the log-concavity of matroidal sequences can be understood as the log-concavity of various specializations of the Tutte polynomial. The most important log-concave matroidal invariant for the purposes of this paper is the following result of Huh.
\begin{theorem} \cite{Huh.2015} \label{thm:lcmatroid}
Let $M$ be any representable matroid over a field $\mathbb{F}$ of characteristic zero. Then $T_{M}(x,1)$ is log-concave.
\end{theorem}

\begin{remark} 
In fact, Ardila, Denham, and Huh ~\cite{ArdDenHuh.2023} have proven that $T_{M}(x,1)$ is log-concave for all matroids $M$. However, we will not need this full generalization as we will be able to derive the log-concavity of parking functions from Huh's original result.
\end{remark}

\section{Results}
\label{section.results}

In Section~\ref{section.main}, we give of a proof of Conjecture ~\ref{conj.main}. In Section~\ref{section.gpark}, we introduce the notion of a $G$-parking function and prove their log-concavity.
\subsection{Log-concavity of parking functions}
\label{section.main}

We now prove the log-concavity of parking functions graded by the $\tsum$ statistic.

\begin{theorem} \label{thm:main}
	The polynomial $P_n(x)$ is log-concave for all $n\geq 1$.
\end{theorem}

\begin{proof}
	By Equation ~\ref{eq:pftree}, it suffices to prove that the polynomial $I_{n}(y)$ is log-concave for all $n\geq 1$. The polynomial $I_n(y)$ is known to be equivalent to $T_{M(K_n)}(1,y)$ where $K_n$ is the complete graph on $n$ vertices ~\cite[13f]{Biggs.1993}. This is then equivalent to $T_{M(K_n)^{\perp}}(y,1)$ via Equation ~\ref{eq:dual}. As $M(K_n)$ is representable over every field, so is its dual $M(K_n)^{\perp}$. Thus, $I_n(y) = T_{M(K_n)^{\perp}}(y,1)$ is log-concave by Theorem ~\ref{thm:lcmatroid} which gives us the desired result.
\end{proof}

Given a labeled graph $G$, let $e(G)$ denote the number of edges in $G$ and let $C_n(x) = \sum_{G} x^{e(G)}$ where the sum ranges over all labeled, connected simple graphs with $n$ vertices. We obtain the following result as a corollary.

\begin{corollary}
	The polynomial $C_n(x)$ is log-concave for all $n\geq1$.
\end{corollary}

\begin{proof}
	The polynomial $C_n(x)$ is related to $I_n(x)$ via the equation $C_n(x) = x^{n-1}I_n(1+x)$, see ~\cite{GesselWang.1979} for a combinatorial proof of this identity. Thus, it suffices to show that $I_n(1+x)$ is log-concave. Given a polynomial $f(x) = a_0 + a_1x\ldots + a_mx^m$ over the nonnegative reals with no internal zeros (i.e. with no three indices $i < j < k$ such that $a_i, a_k \not = 0$ and $a_j = 0$), the shifted polynomial $f(1+x)$ is also log-concave ~\cite{Brenti.1994}. Given that $I_n(x)$ is log-concave with no internal zeros, we have $C_n(x)$ is also log-concave.
\end{proof}

\subsection{Log-concavity of \texorpdfstring{$G$}{G}-parking functions}
\label{section.gpark}

A $G$-parking function is a generalization of parking functions introduced by Postnikov and Shapiro ~\cite{PostSha.2004}. Let $G$ be an undirected, connected graph on the vertex set $\{0, 1, \ldots, n\}$ with possible multiedges, but no loops. Given $I \subseteq \{1, 2, \ldots, n\}$, let $d_I(i)$ be the number of edges between vertex $i$ and vertices not in the subset $I$. A \defn{$G$-parking function} is a sequence $(a_1, a_2, \ldots, a_n)$ of positive integers such that for any nonempty subset $I\subseteq \{1,2,\ldots, n\}$, there exists $i\in I$ such that $a_i \leq d_I(i)$. Note that ordinary parking functions of length $n$ are exactly the $G$-parking functions for the complete graph $K_{n+1}$ on the vertex set $\{0, 1, \ldots, n\}$. 

As before with ordinary parking functions, we will grade $G$-parking functions by their sum and denote their associated generating function by $P_{G}(x)$. This polynomial is essentially equivalent to the level generating function for the abelian sandpile model/chip firing game, see ~\cite{Yan.2015} for details. Thus, the following result also holds for the level enumerator of the abelian sandpile model.

\begin{theorem}
Let $G$ be an undirected, loopless, connected graph. Then $P_{G}(x)$ is log-concave.
\end{theorem}

\begin{proof}
	We have $T_{M(G)}(1,y) = y^{e(G)}P_{G}(y^{-1})$ by ~\cite{CorLeBo.2003, Lopez.1997} where $e(G)$ is the number of edges in $G$. Hence, by the same argument as Theorem ~\ref{thm:main}, we have $P_{G}(x)$ is log-concave.
\end{proof}

\bibliographystyle{alpha}
\bibliography{pf_lc}{}

\end{document}